\documentclass[11pt]{article}

\usepackage{todonotes}

\usepackage[top=3.5cm,bottom=3.5cm,left=3.5cm,right=3cm]{geometry}
\pdfpagewidth\paperwidth
\pdfpageheight\paperheight
\usepackage[english,italian]{babel}
\usepackage[T1]{fontenc}
\usepackage[latin1]{inputenc}
\usepackage{lmodern}
\usepackage{amsfonts}
\usepackage{cancel}
\usepackage{amsmath,amsthm,mathtools}
\usepackage{amssymb}
\usepackage{graphicx}
\usepackage{sidecap}
\usepackage{caption}
\usepackage{subfig}
\usepackage{wrapfig}
\usepackage{psfrag}
\usepackage{mathrsfs}
\usepackage{tikz}
\usepackage{multicol}
\usepackage{pgfplots}
\usepackage{hyperref,enumitem}
\usepackage{ulem}

\pgfplotsset{compat=1.17}

\newtheorem{teo}{Theorem}
\newtheorem{lem}[teo]{Lemma}
\newtheorem{prop}[teo]{Proposition}
\newtheorem{cor}[teo]{Corollary}
\theoremstyle{definition}
\newtheorem{defn}[teo]{Definition}
\newtheorem{ex}[]{Example}

\theoremstyle{remark}
\newtheorem{oss}[teo]{Remark}
\newtheorem*{oss*}{Remark}

\newcommand{\N}{\mathbb N}
\newcommand{\R}{\mathbb R}
\newcommand{\C}{\mathbb C}
\newcommand{\T}{\mathbb T}

\newcommand{\cH}{{\mathcal H}}

\newcommand{\cW}{{\mathcal W}}

\newcommand{\email}[1]{\protect\href{mailto:#1}{#1}}

\newcommand{\D}{{\rm d}}
\newcommand{\E}{{\rm e}}

\definecolor{eggplant}{HTML}{54414E}
\definecolor{terracotta}{HTML}{E07A5F}
\definecolor{queenblue}{HTML}{4D7298}
\definecolor{greensheen}{HTML}{81B29A}
\definecolor{deepchampagne}{HTML}{F2CC8F}
\definecolor{darkpurple}{HTML}{6A0DAD}

\begin{document}
\selectlanguage{english}
\title{Phase retrieval of bandlimited functions for the wavelet transform}
\date{} 

\author{Rima Alaifari\thanks{Seminar for Applied Mathematics, ETH Z{\"u}rich, R{\"a}mistrasse 101, 8092 Z{\"u}rich, Switzerland (\email{rima.alaifari@sam.math.ethz.ch}).}
\and Francesca Bartolucci\thanks{Seminar for Applied Mathematics, ETH Z{\"u}rich, R{\"a}mistrasse 101, 8092 Z{\"u}rich, Switzerland (\email{francesca.bartolucci@sam.math.ethz.ch}).)
}
\and Matthias Wellershoff\thanks{Seminar for Applied Mathematics, ETH Z{\"u}rich, R{\"a}mistrasse 101, 8092 Z{\"u}rich, Switzerland (\email{matthias.wellershoff@sam.math.ethz.ch}).}}

\maketitle

\normalem

\abstract{We study the recovery of square-integrable signals from the absolute values of their wavelet transforms, also called \emph{wavelet phase retrieval}. We present a new uniqueness result for wavelet phase retrieval. To be precise, we show that any wavelet with finitely many vanishing moments allows for the unique recovery of real-valued bandlimited signals up to global sign. Additionally, we present the first uniqueness result for \emph{sampled} wavelet phase retrieval in which the underlying wavelets are allowed to be complex-valued and we present a uniqueness result for phase retrieval from sampled Cauchy wavelet transform measurements.}

\smallskip
\noindent \textit{Key words.} Phase retrieval, Wavelet transform, Paley--Wiener space, Wavelet system, Sampling theorem
%\vspace{2mm}

%\noindent\textit{Mathematics Subject Classification}.
\section{Introduction}

In the present paper, we study the \emph{(continuous) wavelet transform} of signals $f \in L^2(\R)$ associated to a wavelet $\psi \in L^1(\R)$ which is defined by 
\begin{equation*}
    \mathcal{W}_{\psi}f(b,a) := \frac{1}{a} \int_{\R} f(x) \overline{\psi\left(\frac{x-b}{a}\right)} \,\mathrm{d} x, \qquad a \in \R_+,~b \in \R. 
\end{equation*}
In particular, we are interested in the recovery of $f$ from the magnitude-only measurements $\lvert \mathcal{W}_\psi f \rvert$. This problem is typically called \emph{wavelet phase retrieval} and has recently received an increasing amount of attention \cite{alaifari2017reconstructing,holighaus2019characterization,jaming2014uniqueness,mallat2015phase,waldspurger17}. Wavelet phase retrieval can be used in audio analysis and processing. In particular, it allows the experimenter to freely modify the scalogram (a term used to refer to the absolute value of the wavelet transform) of an audio signal and then synthesise the modified scalogram to obtain a modified audio signal. This technique can be applied in blind source seperation and audio texture synthesis for example \cite{bruna2013audio,virtanen2007monaural,waldspurger17}.

It is important to note from the outset that the signals $f$ and $\tau f$, where $\tau \in \mathbb{T}$, generate the same measurements $\lvert \mathcal{W}_\psi f \rvert = \lvert \mathcal{W}_\psi (\tau f) \rvert$ and can thus not be distinguished from wavelet magnitudes only. For this reason, it is customary to ask whether a given signal $f \in L^2(\R)$ can be recovered \emph{up to global phase} from phaseless measurements, i.e.~whether one may recover $f$ up to equivalence defined by the relation
\begin{equation*}
    f \sim g ~ : \Leftrightarrow ~ \exists \, \tau \in \mathbb{T}: g = \tau f.
\end{equation*}
Stated concisely, our aim is therefore to study the injectivity of the \emph{phase retrieval operator} $\mathcal{A}_\psi : \mathcal{M} / \sim \, \to [0,+\infty)^\Lambda$ given by 
\begin{equation*}
    \mathcal{A}_\psi(f)(b,a) := \lvert \mathcal{W}_\psi f (b,a) \rvert, \qquad (b,a) \in \Lambda,
\end{equation*}
where $\Lambda \subseteq \R \times \R_+$ and $\mathcal{M} \subseteq L^2(\R)$.

Determining for which wavelets $\psi \in L^1(\R)$ and which choices of $\Lambda \subseteq \R \times \R_+$ as well as $\mathcal{M} \subseteq L^2(\R)$ the operator $\mathcal{A}_{\psi}$ is injective is a famously difficult problem that has been solved only in very few cases \cite{alaifari2017reconstructing,jaming2014uniqueness,mallat2015phase}. We also refer to \cite{waldspurger17} where the author gives an overview of the wavelet phase retrieval problem and states that \emph{``(...) a theoretical study of the well-posedness, for relatively general wavelets, seems out of reach''}. 

\paragraph{Prior work} Let us highlight some prior work which is related to the present paper: Firstly, we want to mention \cite{alaifari2017reconstructing} in which the authors study a \emph{wavelet sign retrieval} problem. In particular, they consider a setup in which both the signal and the wavelet are assumed to be real-valued such that the wavelet coefficients are real-valued as well. In this case, the problem of recovering $f$ from $\lvert \mathcal{W}_\psi f \rvert$ amounts to recovering the sign of $\mathcal{W}_\psi f$. 

Secondly, we want to mention \cite{mallat2015phase} in which the authors prove injectivity of the operator $\mathcal{A}_\psi : \mathcal{H}_+ / \sim \, \to [0,+\infty)^{\R \times \{1,a\}}$, where $a > 1$ and 
\[
    \mathcal{H}_+ := \{ f \in L^2(\R) \,;\, \operatorname{supp} \widehat f \subset \R_+ \}
\]
is the space of \emph{analytic signals}. They do so for a specific family of \emph{progressive\footnote{Here and throughout this paper, we will call a wavelet $\psi \in L^1(\R)$ progressive if it only has positive frequencies, i.e.~$\psi \in \mathcal{H}_+$.} wavelets} called the \emph{Cauchy wavelets}. In other words, the authors of \cite{mallat2015phase} show that the magnitude of the Cauchy wavelet transform of any signal $f \in L^2(\R)$ uniquely determines its \emph{analytic representation} $f_+ \in L^2(\R)$, given by 
\begin{equation*}
    \widehat f_+ (\xi) := 2 \widehat f (\xi) \boldsymbol{1}_{\xi > 0}, \qquad \xi \in \R,
\end{equation*}
up to global phase. It is worthwhile pointing out that the negative frequencies of the signal $f$ are inevitably lost in the measurement process as the wavelet transform with a progressive wavelet satisfies 
\begin{align*}
    \mathcal{W}_\psi f (b,a) &= \int_\R \widehat f (\xi) \overline{\widehat \psi (a\xi)} \mathrm{e}^{2\pi\mathrm{i}\xi b} \, \mathrm{d} \xi = \int_0^\infty \widehat f (\xi) \overline{\widehat \psi (a\xi)} \mathrm{e}^{2\pi\mathrm{i}\xi b} \, \mathrm{d} \xi \\ &= \frac12 \int_0^\infty \widehat f_+ (\xi) \overline{\widehat \psi (a\xi)} \mathrm{e}^{2\pi\mathrm{i}\xi b} \, \mathrm{d} \xi = \frac12 \mathcal{W}_\psi f_+ (b,a),
\end{align*}
where the first equality is due to a direct application of Plancherel's theorem. The results in \cite{mallat2015phase} are thus optimal in the sense that one cannot hope to recover the negative frequencies of $f \in L^2(\R)$ from its Cauchy wavelet transform $\mathcal{W}_\psi f$. Since wavelet phase retrieval (for one-dimensional signals) appears to be almost exclusively useful to practitioners of audio processing, it is tempting to argue that this limitation is not really important: Indeed, in audio processing one is predominantly interested in real-valued signals and the negative frequencies of real-valued signals $f \in L^2(\R)$ are uniquely determined by the analytic representation $f_+$ through the relation
\begin{equation*}
    \widehat f (-\xi) = \overline{\widehat f(\xi)} = \frac{\overline{\widehat f_+ (\xi)}}{2}, \qquad \xi > 0.
\end{equation*}

In phase retrieval, however, this observation only tells part of the story because determining the analytic representation $f_+$ of a real-valued signal $f \in L^2(\R)$ up to global phase does \emph{not} amount to retrieving the real-valued signal up to global sign as one could hope: Indeed suppose that $f,g \in L^2(\R)$ are real-valued and satisfy $g_+ = \tau f_+$, for some $\tau \in \T$. Then, it follows that 
\begin{equation*}
    \widehat g (-\xi) = \frac{\overline{\widehat g_+ (\xi)}}{2}
    = \overline{\tau} \frac{\overline{\widehat f_+ (\xi)}}{2}
    = \overline{\tau} \widehat f(-\xi), \qquad \xi > 0.
\end{equation*}
One may thereby see that $\widehat f$ and $\widehat g$ do not necessarily agree up to global phase and it follows that $f$ and $g$ do not necessarily agree up to global sign (for more details see also Remark \ref{rem:rimas_counterexample}).

\paragraph{Contributions} In this paper, we present a new uniqueness result for wavelet phase retrieval in which the underlying wavelets are allowed to be \emph{complex-valued}. In contrast to \cite{alaifari2017reconstructing}, we are thus considering a phase recovery problem in which the measurements are complex-valued in general. Furthermore, our results guarantee the unique recovery of the real-valued signals themselves, instead of their analytic representations merely.

To the best of our knowledge our result is one of the first uniqueness results for wavelet phase retrieval in the literature (apart from the already mentioned \cite{alaifari2017reconstructing,jaming2014uniqueness,mallat2015phase}). It is an attempt to partially answer the conjecture in \cite{waldspurger17} that \emph{``(...) the inverse problem\footnote{By ``the inverse problem'' the phase retrieval problem is meant. The corresponding direct problem is to evaluate the phase retrieval operator $\mathcal{A}_\psi$.} is well-posed for generic wavelet families.''} To be precise, we develop the following result:

\begin{teo}[Cf.~Theorem \ref{thm:unquenessresultreal}]
    \label{thm:mainthmI}
    Let $\psi \in L^1(\R)$ be a wavelet with finitely many vanishing moments. Then, any real-valued \emph{bandlimited} function $f \in L^2(\R)$ is uniquely determined by $\lvert \mathcal{W}_\psi f \rvert$ up to global sign. 
\end{teo}

Restriction of the domain of $\mathcal A_\psi$ to real-valued bandlimited functions is motivated by audio processing in which this is typically a reasonable assumption.

Apart from Theorem \ref{thm:mainthmI}, we also present the first uniqueness results for \emph{sampled} wavelet phase retrieval in which the underlying wavelets are allowed to be complex-valued. Inter alia, we are able to prove the following result:

\begin{teo}[Cf.~Theorem \ref{thm:unquenessresultrealsampled2}]
    \label{thm:mainthmII}
    Let $a > 1$, $b > 0$ and let $\psi \in L^1(\R)$ be a wavelet with finitely many vanishing moments. Then, any real-valued bandlimited function $f \in L^2(\R)$ is uniquely determined by the measurements $\lvert \mathcal{W}_\psi f \rvert$ on the discrete set $a^{-\N}(b \mathbb{Z} \times \{1\})$ up to global sign. 
\end{teo}

Note that the measurements $\lvert \mathcal{W}_\psi f \rvert$ on the discrete set $a^{-\N}(b \mathbb{Z} \times \{1\})$ exactly correspond to the phaseless wavelet coefficients at fine scales (i.e.~$k \geq 1$) for the wavelet system
\begin{equation*}
    \mathcal{W}(\psi,a,b) := \{ a^{-k} \psi(a^{-k} \cdot - bm) \}_{k,m \in \mathbb{Z}}.
\end{equation*}
Our result is therefore compatible with the classical theory on wavelet frames \cite{daub92,heil2011basis}. In addition, we want to point out that our result is reminiscent of the recent uniqueness results for Gabor phase retrieval from samples \cite{alaifariwell20,grohs2020injectivity}.

Finally, we apply some of the insights used in the proofs of the Theorems \ref{thm:mainthmI} and \ref{thm:mainthmII} to reconsider the work presented in \cite{mallat2015phase}. In this way, we are able to prove that the analytic representation of a bandlimited signal is uniquely determined by the magnitude of its Cauchy wavelet transform on the set $b \mathbb{Z} \times \{1,a\}$ up to global phase, where $2 b > 0$ is upper bounded by the Nyquist rate and $a > 1$ (cf.~Theorem \ref{thm:cwtmaintheorem}).

%\paragraph{Open problems} It does not seem possible to generalise the Theorems \ref{thm:mainthmI} and \ref{thm:mainthmII} to \emph{complex-valued} bandlimited functions (or even to all of $L^2(\R)$) by using the techniques we employ in our proofs. We are of the opinion that such a generalisation would be highly interesting from the point of view of pure mathematics (while potentially being less riveting for application purposes). 
%
%It is of course also interesting to study whether one can reduce the amount of measurements needed to recover real-valued bandlimited signals from sampled wavelet transform magnitudes with respect to Theorem \ref{thm:mainthmII}. In particular, it seems conceivable that magnitude information on a finite number of scales (potentially as few as two scales) is sufficient for recovery up to global sign. The result on the Cauchy wavelet transform provides some evidence in this direction: Indeed, one can prove a uniqueness result for the analytic part of bandlimited signals by only making use of two scales (cf.~Theorem \ref{thm:cwtmaintheorem}).

\paragraph{Outline} In Section~\ref{sec:preliminaries}, we recall the definition of the wavelet transform and of the Paley--Wiener space. Moreover, we prove some auxiliary results that are needed later in the paper. In  Section~\ref{sec:main_results}, we state and prove Theorems \ref{thm:mainthmI} and \ref{thm:mainthmII}. In Section~\ref{sec:examples}, we apply Theorems \ref{thm:mainthmI} and \ref{thm:mainthmII} to the Morlet wavelet and the chirp wavelet. Finally, in Section~\ref{sec:cwt}, we consider the results from \cite{mallat2015phase} and prove a sampling result for the Cauchy wavelet transform.

\paragraph{Notation} We set $\R_+ := (0,+\infty)$. For any $p\in[1,+\infty]$ we denote by $L^p(\R)$ the Banach space of functions $f:\R\rightarrow\C$ which are $p$-integrable with respect to the Lebesgue measure and we use the notation $\|\cdot\|_p$ for the corresponding norms. The Fourier transform on $L^1(\R)$ is defined by
\[
\widehat{f}(\xi) := \int_{\R} f(x) \E^{-2\pi \mathrm{i} 
  x \xi } \,\mathrm{d} x, \qquad \xi \in \R,
\] 
and it extends to $L^2(\R)$ by a classical density argument. Finally, for $\ell \in \N$ and any sufficiently smooth function $f$, we denote by $f^{(\ell)}$ the $\ell$-th derivative.

\section{Preliminaries}\label{sec:preliminaries}
The translation and dilation operators act on a
function $f:\R\to \C$ by
\[
T_bf(x)=f(x-b),
\qquad
D_af(x)=a^{-1}f\left(a^{-1}x\right),
\]
respectively, for every $b\in\R$ and $a\in\R_+$. Both operators
map $L^p(\R)$ onto itself and $D_a$ is normalized to be an isometry on $L^1(\R)$. For every $a\in\R_+$, we will use the notation $f_a=D_af$. Furthermore, let us denote $f^\#(x)=\overline{f(-x)}$, $x\in\R$. 
\begin{defn}\label{defn:wavelettransform}
Let $1\leq p\leq\infty$. The wavelet transform of $f\in L^p(\R)$ associated with $\psi\in L^1(\R)$ is defined by 
\begin{equation}\label{eq:waveletcoefficients}
\mathcal{W}_{\psi}f(b,a)= (f *\psi_a^\#)(b)=\frac{1}{a}\int_{\R}f(x)\overline{\psi\left(\frac{x-b}{a}\right)}\D x,
\end{equation}
for every $b\in\R$ and $a\in\R_+$. 
\end{defn}
We observe that, by Young's inequality, $\mathcal{W}_{\psi}f(\cdot,a)\in L^p(\R)$ for every $a\in\R_+$. We refer also to \eqref{eq:waveletcoefficients} as the wavelet coefficient of $f$ at $(b,a)$ with respect to $\psi$. In this context, a non-zero function $\psi\in L^1(\R)$ is called a wavelet if $\widehat{\psi}(0)=0$, or equivalently if 
\begin{equation}\label{eq:zeromeancondition}
\int_{\R}\psi(x){\rm d}x=0.
\end{equation}
The name ``wavelet'' refers to the fact that condition \eqref{eq:zeromeancondition} forces such functions to have some oscillations. It is well known that if $\psi\in L^2(\R)$ is a progressive
wavelet, i.e.~a wavelet with only positive frequencies, satisfying the \emph{admissibility condition}
\[
0<\int_0^{+\infty} \frac{|\widehat{\psi}(\xi)|^2}{\xi}{\rm d}\xi<+\infty,
\]
then the wavelet transform is a constant multiple of an isometry from the analytic signals $\cH_+$ into $L^2(\R\times\R_+,a^{-1}{\rm d}b{\rm d}a)$ (see e.g. \cite[Theorem 22.0.6]{hol95}). It follows in particular that the wavelet transform is injective under the above conditions.

We fix $\Omega>0$ and we denote by $\text{PW}_{\Omega}$ the space of bandlimited functions 
\[
\text{PW}_{\Omega}=\{f\in L^2(\R):\text{supp}\widehat{f}\subseteq[-\Omega,\Omega]\}
\]
which is a closed subspace of $L^2(\R)$. By the Paley--Wiener theorem, every $f\in \text{PW}_{\Omega}$ has an analytic extension to an entire function of exponential type which we also denote by $f$. More precisely,
\[
|f(z)|\leq\frac{1}{\sqrt{2\pi}}\|\widehat{f}\|_1 \text{e}^{|\text{Im}z|\Omega},\quad z\in\mathbb{C}.
\]
We can therefore consider $\text{PW}_{\Omega}$ to be a Hilbert space of entire functions. Furthermore, the space of bandlimited functions $\text{PW}_{\Omega}$ is a reproducing kernel Hilbert space (RKHS) (see for example \cite[Chapter 2]{daub92}). This means that for every $x\in\R$ the evaluation operator $L_x\colon \text{PW}_{\Omega}\to\C$ defined by 
\[
L_x(f)=f(x),\qquad f\in\text{PW}_{\Omega},
\]
is bounded. Therefore, if $f_n$ is a sequence in $\text{PW}_{\Omega}$ which converges to $f$ in $L^2(\R)$ as $n\to\infty$, then 
\[
f_n(x)\to f(x),\quad n\to\infty,
\]
for every $x\in\R$. The next lemma will play a crucial role in the proof of our main results. It follows immediately from Theorem 1 on p.~723 of \cite{thakur10}.

\begin{lem}\label{lem:uniquenessmodulus}
Let f be an entire function real-valued on the real line. Then, $f$ is uniquely determined by $\{|f(x)|:x\in\R\}$ up to global sign. 
\end{lem}
 
It is worth observing that if $f\in\text{PW}_{\Omega}$ and $\psi\in L^1(\R)$, then $\mathcal{W}_{\psi}f(\cdot,a)$ is also a bandlimited function for every $a\in\R_+$:

\begin{lem}\label{lem:wavelettransformbandlimited}
Let $a \in \R_+$ and $\Omega>0$. Furthermore, let $f\in\text{PW}_{\Omega}$ and $\psi\in L^1(\R)$. Then, we have that $\mathcal{W}_{\psi}f(\cdot,a)\in \text{PW}_{\Omega}$.
\end{lem}

\begin{proof}
It follows from the convolution theorem that $f*\psi_a^\#\in L^2(\R)$ and 
\begin{equation}\label{eq:convolutiontheorem}
 (f*\psi_a^\#)^\wedge(\xi)=\hat{f}(\xi)(\psi_a^\#)^\wedge(\xi),\qquad \xi\in\R.
 \end{equation}
Moreovoer, by the relation
\[
(\psi_a^\#)^\wedge(\xi)=\overline{\widehat{\psi}(a\xi)},\qquad \xi\in\R,
\]
equation \eqref{eq:convolutiontheorem} becomes  
 \[
 (f*\psi_a^\#)^\wedge(\xi)=\widehat{f}(\xi)\overline{\widehat{\psi}(a\xi)},\qquad \xi\in\R.
 \]
Then, since $\text{supp}\ (f*\psi_a^\#)^\wedge\subseteq \text{supp}\ \widehat{f}$ and $f\in\text{PW}_{\Omega}$, we conclude that $\mathcal{W}_{\psi}f(\cdot,a)\in \text{PW}_{\Omega}$.
\end{proof}

%Lemma~\ref{lem:wavelettransformbandlimited} follows 
%immediately by the convolution theorem and by the relation
%\[
%(\psi_a^\#)^\wedge(\xi)=\overline{\widehat{\psi}(a\xi)},\quad a\in\R_+,\, %\xi\in\R.
%\]
% \begin{prop}[Convolution theorem]\label{prop:convolutiontheorem}
% If $f\in L^2(\R)$ and $g\in L^1(\R)$, then $f*g\in L^2(\R)$ and
% \[
% (f*g)^\wedge(\xi)=\hat{f}(\xi)\hat{g}(\xi),\quad \xi\in\R.
% \]
% Furthermore, if $f\in L^2(\R)$ and $g\in L^2(\R)$, then $fg\in L^1(\R)$ and 
% \[
% (fg)^\wedge(\xi)=(\hat{f}*\hat{g})(\xi),\quad \xi\in\R.
% \]
% \end{prop}
% \raadd{Rima: I suggest that we can remove Proposition 4, since the conv. theorem should be standard.}
% \mcwadd{Matthias: I agree. Then, we also need to remove the part that I colored in orange above.} \raadd{Rima: Yes, I agree.}

A first insight into wavelet phase retrieval comes from approximation theory.

\begin{defn}
An approximate identity is a family $\{\phi_\epsilon\}_{\epsilon\in\R_+}$ of functions in $L^1(\R)$ such that 
\begin{itemize}
\item[i)] $\int_\R\phi_{\epsilon}(x)\D x=1$ for every $\epsilon>0$,
\item[ii)] $\sup_{\epsilon>0}\lVert\phi_{\epsilon}\rVert_1<+\infty$,
\item[iii)] for every $\delta>0$,
\[
\lim_{\epsilon\to0}\int_{|x|\geq\delta}|\phi_{\epsilon}(x)|\D x=0.
\]
\end{itemize}
\end{defn}

\begin{ex}\label{ex:dilationapproidentity}
Let $\phi\in L^1(\R)$ be such that 
\[
\int_{\R}\phi(x)\D x=1,
\]
or equivalently $\widehat{\phi}(0)=1$,
and let $\phi_a(x)=a^{-1}\phi(a^{-1}x)$.
Then, the family of functions $\{\phi_a\}_{a\in\R_+}$ forms an approximate identity. 
\end{ex}
Let $1\leq p<\infty$. It is a well-known fact that the convolution $f*\phi_\epsilon$ converges to $f$ in the $L^p$-norm for every $f\in L^p(\R)$ (see for instance Theorem 1.2.19 on p.~27 of \cite{grafakos2014classical} or consider the part on approximate identities in the classical books \cite{simon,rudin91}):
\begin{prop}\label{prop:convergencenormp}
Let $\{\phi_{\epsilon}\}_{\epsilon\in\R_+}$ be an approximate identity and $1\leq p<\infty$. Then, $f*\phi_\epsilon\in L^p(\R)$ for every $f\in L^p(\R)$  and $\epsilon\in\R_+$. Moreover,
\[
\lim_{\epsilon\to0^+}\|f-f*\phi_\epsilon\|_p=0.
\]
\end{prop}

Proposition~\ref{prop:convergencenormp} together with Lemma~\ref{lem:uniquenessmodulus} implies that, given an approximate identity $\{\phi_{\epsilon}\}_{\epsilon\in\R_+}$, any real-valued $f\in {\rm PW}_{\Omega}$ can be uniquely recovered (up to a global sign factor) from the measurements $\{|f*\phi_\epsilon|\}_{\epsilon\in\R_+}$: 

\begin{teo}\label{thm:unquenessresultrealfirst}
Let $\{\phi_{\epsilon}\}_{\epsilon\in\R_+}$ be an approximate identity. Then, the following are equivalent for $f,g\in {\rm PW}_{\Omega}$ real-valued on the real line:
\begin{itemize}
\item[i)] $|f * \phi_\epsilon|=|g * \phi_\epsilon|,\quad \epsilon\in\mathbb{R}_+$;
\item[ii)] $f=\pm g$.
\end{itemize}
\end{teo}
\begin{proof}
It is clear that if $f=\pm g$, then $i)$ holds. Conversely, we suppose that $|f * \phi_\epsilon|=|g * \phi_\epsilon|$ for every $\epsilon\in\mathbb{R}_+$. By Proposition~\ref{prop:convergencenormp}, we have that $f*\phi_\epsilon$ and $g*\phi_\epsilon$ converge to $f$ and $g$ in $L^2(\R)$ as $\epsilon\to0^+$, respectively. Since by the convolution theorem $f*\phi_\epsilon$ and $g*\phi_\epsilon$ belong to $\text{PW}_{\Omega}$ for every $\epsilon\in\R_+$ and $\text{PW}_{\Omega}$ is a RKHS, we have that $f*\phi_\epsilon$ and $g*\phi_\epsilon$ converge to $f$ and $g$ pointwise as $\epsilon\to0^+$. Furthermore, since the modulus is a continuous function, $|f*\phi_\epsilon|$ and $|g*\phi_\epsilon|$ converge pointwise to $|f|$ and $|g|$ as $\epsilon\to0^+$. Therefore, our assumption implies $|f(x)|=|g(x)|$ for every $x\in\R$. Hence, by Lemma~\ref{lem:uniquenessmodulus}, we can conclude that $f=\pm g$.
\end{proof}

Let us fix $\psi \in L^1(\R)$ such that $\widehat \psi(0) = 1$. It follows from the definition of the wavelet transform (cf.~Definition \ref{defn:wavelettransform}) together with the considerations in Example \ref{ex:dilationapproidentity} and Theorem \ref{thm:unquenessresultrealfirst} that any real-valued $f\in{\rm PW}_\Omega$ can be uniquely recovered (up to a global sign factor) from the magnitude of its wavelet transform $\{|f*\psi_a^\#|\}_{a\in\mathbb{R}_+}$.
Unfortunately, we cannot apply Theorem~\ref{thm:unquenessresultrealfirst} when $\psi$ is a classical wavelet since wavelets are always assumed to have zero mean. It is therefore natural to ask if it is possible to recover the same uniqueness result
when $\widehat{\psi}(0)=0$. The next section is devoted to answering this question. 

\section{Main results}
\label{sec:main_results}

We say that a function $\psi\in L^1(\R)$ has $n$ vanishing moments, for $n \in \mathbb{N}$, if it satisfies
\begin{equation}\label{eq:vanishingmoments}
\int_{\R}x^\ell\psi(x)\,\mathrm{d} x=0,\qquad \ell=0,\ldots,n. 
\end{equation}
By the definition of the Fourier transform, condition~\eqref{eq:vanishingmoments} with $n=0$ is equivalent to $\widehat{\psi}(0)=0$. In general, we have the following result:
\begin{prop}[{\cite[Lemma 6.0.4]{hol95}}]
Let $n\in\N$ and $\psi\in L^1(\R)$ be such that $x^n\psi\in L^1(\R)$. Then, $\psi$ has $n$ vanishing moments if and only if 
\[
\lim_{\xi\to0} \xi^{-n}\widehat{\psi}(\xi)=0.
\] 
\end{prop}

%The next proposition is a classical result, see e.g. \cite[Chapter 4, \S %2]{hol95}. Here, we give an alternative proof for the sake of completeness. %The next proposition is a classical result, see e.g. \cite[Chapter 4, \S %2]{hol95}. 
We say that a function $\psi$ has a \emph{finite number of vanishing moments} if there exists an $\ell\in\N$ such that 
\begin{equation*}
    \lim_{\xi\to0}\xi^{-\ell}\widehat{\psi}(\xi)\in \C\setminus\{0\}.
\end{equation*}
It is a well-known fact that if we choose a wavelet with a finite number of vanishing moments, the wavelet transform approximates the derivatives of a smooth signal at fine scales, see e.g.~\cite[Chapter 6]{mallat98} or \cite[Chapter 4, \S 2]{hol95}. We are interested in a different setup than the one chosen in the references mentioned before and therefore present the following proposition and its proof.
\begin{prop}\label{thm:centraletheorem}
Let $\Omega>0$ and let $\psi\in L^1(\R)$ be such that 
\begin{equation*}
    \lim_{\xi\to0}\xi^{-\ell}\widehat{\psi}(\xi)= (-1)^\ell(2\pi i)^\ell,
\end{equation*}
for some $\ell \in \N$.
Then, for every $f\in{\rm PW}_\Omega$ 
\[
\lim_{a\to0^+}a^{-\ell}\mathcal{W}_{\psi}f(b,a)=f^{(\ell)}(b),\qquad b\in\R.
\]
\end{prop}
\begin{proof}
By the definition of the wavelet transform and the Plancherel theorem, we have
\begin{align*}
\|f^{(\ell)}-a^{-\ell}\mathcal{W}_{\psi}f(\cdot,a)\|_2^2&=\|f^{(\ell)}-a^{-\ell}f * \psi_a^\#\|_2^2\\
&=\int_{\R}|\xi^\ell \widehat{f}(\xi)|^2|(2\pi i)^\ell-(a\xi)^{-\ell}\overline{\widehat{\psi}(a\xi)}|^2{\rm d}\xi.
\end{align*}
By the Riemann--Lebesgue lemma, $\widehat{\psi}$ is a continuous function which goes to zero at infinity and by hypothesis
\[
\lim_{\xi\to0}\frac{\widehat{\psi}(\xi)}{\xi^\ell}= (-1)^\ell(2\pi i)^\ell.
\]
Therefore, we have the estimate
\[
|\xi^\ell\widehat{f}(\xi)|^2|(2\pi i)^\ell-(a\xi)^{-\ell}\overline{\widehat{\psi}(a\xi)}|^2\leq M |\xi^\ell\widehat{f}(\xi)|^2,
\]
where $M=\sup_{\xi\in\R} |(2\pi i)^\ell-(a\xi)^{-\ell}\overline{\widehat{\psi}(a\xi)}|^2$ is finite and independent of $a$. Furthermore, for almost every $\xi\in\R$ it holds that
\[
\lim_{a\to0^+}|\xi^\ell\widehat{f}(\xi)|^2|(2\pi i)^\ell-(a\xi)^{-\ell}\overline{\widehat{\psi}(a\xi)}|^2=0.
\]
Hence, by the dominated convergence theorem
\[
\lim_{a\to0^+}\|f^{(\ell)}-a^{-\ell}\mathcal{W}_{\psi}f(\cdot,a)\|_2=0.
\]
Furthermore, by Lemma~\ref{lem:wavelettransformbandlimited}, we know that  $a^{-\ell}\mathcal{W}_{\psi}f(\cdot,a)$ belongs to $\text{PW}_{\Omega}$ for every $a\in\R_+$ and $\text{PW}_{\Omega}$ is a RKHS. Thus, $a^{-\ell}\mathcal{W}_{\psi}f(\cdot,a)$ converges pointwise to $f^{(\ell)}$ as $a\to0^+$, and this concludes the proof. 
\end{proof}

We are now in a position to state our first result establishing uniqueness of wavelet phase retrieval for real-valued bandlimited signals when the wavelet has finitely many vanishing moments.

\begin{teo}[Cf.~Theorem \ref{thm:mainthmI}]
\label{thm:unquenessresultreal}
    Let $\Omega>0$ and let $\psi\in L^1(\R)$ be such that 
    \begin{equation*}
        \lim_{\xi\to0}\xi^{-\ell}\widehat{\psi}(\xi)=c\in\mathbb{C}\setminus\{0\},
    \end{equation*}
    for some $\ell \in \N$.
    Then, the following are equivalent for $f,g\in {\rm PW}_{\Omega}$ real-valued on the real line:
    \begin{itemize}
        \item[i)] $|\mathcal{W}_{\psi}f(b,a)|=|\mathcal{W}_{\psi}g(b,a)|,\quad b\in\R,\, a\in\mathbb{R}_+$;
        \item[ii)] $f=\pm g$.
    \end{itemize}
\end{teo}

\begin{proof}
    Let $f,g\in \text{PW}_{\Omega}$. It is clear that if $f=\pm g$, then $i)$ holds. Conversely, we suppose that $ii)$ holds. 
    Let us define
    \begin{equation*}
        \phi := \frac{(-1)^\ell(2\pi i)^\ell}{c}\psi.   
    \end{equation*}
    Then, we have that  
    \begin{equation*}
        \lim_{\xi\to0}\xi^{-\ell}\widehat \phi(\xi) = (-1)^\ell(2\pi i)^\ell.
    \end{equation*}
    By Proposition~\ref{thm:centraletheorem}, it follows that $a^{-\ell}\mathcal{W}_{\phi}f(\cdot,a)$ and $a^{-\ell}\mathcal{W}_{\phi}g(\cdot,a)$ converge pointwise to $f^{(\ell)}$ and $g^{(\ell)}$, respectively, as $a\to0^+$. Since the absolute value is a continuous function, $a^{-\ell}|\mathcal{W}_{\phi}f(\cdot,a)|$ and $a^{-\ell}|\mathcal{W}_{\phi}g(\cdot,a)|$ converge pointwise to $|f^{(\ell)}|$ and $|g^{(\ell)}|$, respectively, as $a\to0^+$. Combining this with item $i)$ implies that $|f^{(\ell)}(b)|=|g^{(\ell)}(b)|$, for every $b\in\R$, and employing Lemma~\ref{lem:uniquenessmodulus} we can conclude that $f^{(\ell)}=\pm g^{(\ell)}$.
    
    Together with the analyticity of $f$ and $g$ this implies
    \[
    f(x)\mp g(x)=P(x),\quad x\in\R,
    \]
    where $P$ is a polynomial of degree $\ell-1$. Now, if $P$ is not the null polynomial, then $f\mp g$ is not in $L^2(\R)$ and we have a contradiction. Therefore, $P\equiv 0$ and $f=\pm g$.
\end{proof}

\begin{oss}
\label{rem:rimas_counterexample}
It is worth observing that Theorem~\ref{thm:unquenessresultreal} does not hold for progressive wavelets, that is, for wavelets with only positive frequencies. Indeed, the hypothesis of Theorem~\ref{thm:unquenessresultreal} will always be violated since for all progressive wavelets $\psi$ and any $\ell \in \N$
\[
    \lim_{\xi \to 0^-} \xi^{-\ell} \widehat \psi (\xi) = 0.
\]
Actually, one can convince oneself that real-valued signals can never be uniquely determined up to global phase by the magnitude of their wavelet transform with respect to any progressive wavelet. Indeed, by the definition of the wavelet transform, it is immediate to observe that if $f, g\in L^2(\R)$ are such that $f_+=e^{i\alpha}g_+$ and $\psi$ is a progressive wavelet, then 
\begin{equation}\label{eq:equalitymagnitudewavelets}
    |\mathcal{W}_\psi f(b,a)|=|\mathcal{W}_\psi g(b,a)|,
\end{equation}
for all $b\in\R$ and $a\in\R_+$. Additionally, we can show that it is actually possible to construct real-valued signals that do not agree up to global phase even though their analytic representations do, and thus \eqref{eq:equalitymagnitudewavelets} is satisfied. To do so, we consider $f,g \in L^2(\R)$ as well as $\alpha\in\R$ and we suppose that
\begin{equation*}
    f_+=e^{i\alpha}g_+,
\end{equation*}
or equivalently 
\begin{equation*}
    \operatorname{Re} f_+=\text{Re}(e^{i\alpha}g_+),\quad \operatorname{Im} f_+=\text{Im}(e^{i\alpha}g_+).
\end{equation*}
We recall that the analytic representation $f_+$ of $f$ is given by
\begin{equation}\label{eq:analitycparthilberttransform}
 f_+(x)=f(x)+i(\mathscr{H}f)(x), 
\end{equation}
where $\mathscr{H} f$ denotes the Hilbert transform of $f$. 
By equation \eqref{eq:analitycparthilberttransform}, $\operatorname{Re} f_+=\text{Re}(e^{i\alpha}g_+)$ is equivalent to 
\begin{equation}\label{eq:fexpression}
 f=\cos{\alpha} \cdot g-\sin{\alpha} \cdot \mathscr{H}g
\end{equation}
and, analogously, $\operatorname{Im} f_+=\text{Im}(e^{i\alpha}g_+)$ is equivalent to 
\begin{equation}\label{eq:hilbertfexpression}
\mathscr{H}f=\cos{\alpha}\cdot  \mathscr{H}g+\sin{\alpha} \cdot g.  
\end{equation}
Furthermore, the property $\mathscr{H}(\mathscr{H}f)=-f$ implies that equations \eqref{eq:fexpression} and \eqref{eq:hilbertfexpression} are equivalent and thus, $f_+=e^{i\alpha}g_+$ if and only if $f$ takes the form \eqref{eq:fexpression}. Therefore, if we take a real-valued signal $g\in L^2(\R)$, as well as $\alpha\in\R$, and we define $f$ by \eqref{eq:fexpression}, then $f$ is real-valued and $f_+=e^{i\alpha}g_+$. However, if $\alpha$ is not a multiple of $\pi$, then $f$ and $g$ will in general not agree up to global phase.

We remark that if $\psi$ is a progressive wavelet, the wavelet phase retrieval problem continues to be not injective even if we allow the scale $a$ to vary over $\R\setminus\{0\}$. Indeed, in that case, we would have that
\[
\overline{\text{span}\{T_bD_a\psi\}}_{b\in\R, a\in\R_+}=\cH_+\quad\text{and}\quad\overline{\text{span}\{T_bD_a\psi\}}_{b\in\R, a\in\R_-}=\cH_-,
\]
where $\R_-=(-\infty,0)$ and $\mathcal{H}_-=\{f\in L^2(\R): \text{supp} \widehat{f}\subseteq\R_-\}$ . Therefore, the so-called complement property (CP), which is a necessary condition for the injectivity of the operator, see e.g. \cite{alaifari2017phase,balan2006signal,cahill2016phase},
\[
    \mathcal{A}_{\psi}f=(|\mathcal{W}_{\psi}f(b,a)|)_{\R\times\R^{\times}},\quad f\in L^2(\R)/\sim,
\] 
would not be satisfied. 
\end{oss}

\begin{oss}
    The proof of Theorem~\ref{thm:unquenessresultreal} can be applied to other spaces of entire functions which are real-valued on the real line. We mention the class of shift-invariant spaces with Gaussian generator: Let $\varphi_\gamma(t)=e^{-\gamma x^2}$, $\gamma>0$. The shift-invariant space $V^\infty(\varphi)$ generated by the Gaussian $\varphi_\gamma$ is defined as
    \[
    V^\infty(\varphi_\gamma)=\{f\in L^\infty(\R): f=\sum_{k\in\mathbb{Z}}c_k\varphi(\cdot-k),~c\in l^\infty(\mathbb{Z})\}.
    \]
    By \cite[Lemma 4.1]{grochenig2020sharp} every $f\in V^\infty(\varphi_\gamma)$ possesses an extension to an entire function satisfying the growth estimate \begin{equation*}
        |f(x+iy)|\lesssim e^{\gamma y^2}, \qquad x,y \in \R.
    \end{equation*}
    Phase retrieval in shift-invariant spaces has recently been studied in several papers \cite{grochenig2020phase,romero2021sign,grohs2020injectivity}.
\end{oss}

We now introduce the Paley--Wiener space of integrable, bandlimited functions  
\[
\text{PW}^1_{\Omega}=\{f\in L^1(\R):\text{supp}\widehat{f}\subseteq[-\Omega,\Omega]\}
\]
and we observe that $\text{PW}^1_{\Omega}\subseteq\text{PW}_{\Omega}$ for every $\Omega>0$. Furthermore, the following result holds:
\begin{prop}\label{prop:squareinpw}
Let $\Omega>0$ and $f\in {\rm PW}_{\Omega}$. Then, $|f|^2\in{\rm PW}^1_{2\Omega}$.
\end{prop}
\begin{proof}
We first note that $|f|^2=f\overline{f}\in L^1(\R)$ and by the convolution theorem
\[
(|f|^2)^\wedge=\widehat{f}*\widehat{\overline{f}}=\widehat{f}*\widehat{f}^\#.
\]
Then, since $\text{supp}\, (|f|^2)^\wedge\subseteq \text{supp}\,\widehat{f}+\text{supp}\,\widehat{f}^\#\subseteq[-2\Omega,2\Omega]$, we conclude that $|f|^2\in{\rm PW}^1_{2\Omega}$.
\end{proof}
Additionally, we will make use of the Whittaker--Shannon--Kotelnikov (WSK) sampling theorem in the following form:

\begin{teo}[WSK sampling theorem]
\label{thm:wsk}
Let $\Omega>0$ and $f\in{\rm PW}_{\Omega}$. Then, for every $x\in\R$
\[
f(x)=\sum_{m\in\mathbb{Z}}f\left(\frac{m}{2\Omega}\right){\rm sinc}(2\Omega x-m).
\]
\end{teo}

We may now state and prove the following result on sampled wavelet phase retrieval:

\begin{teo}\label{thm:unquenessresultrealsampled}
Let $\Omega>0$ and let $\psi\in L^1(\R)$ be such that 
\begin{equation*}
   \lim_{\xi\to0}\xi^{-\ell}\widehat{\psi}(\xi)=c\in\mathbb{C}\setminus\{0\}, 
\end{equation*}
for some $\ell\in\mathbb{N}$. Furthermore, let $(a_k)_{k\in\N}$ be a sequence in $\R_+$ such that $a_k\to0$ as $k\to\infty$.
Then, the following are equivalent for $f,g\in {\rm PW}_{\Omega}$ real-valued on the real line:
\begin{itemize}
\item[i)] $\left|\mathcal{W}_{\psi}f(\frac{m}{4\Omega},a_k)\right|=\left|\mathcal{W}_{\psi}g(\frac{m}{4\Omega},a_k)\right|,\quad m\in\mathbb{Z},\, k\in\N$;
\item[ii)] $f=\pm g$.
\end{itemize}
\end{teo}
\begin{proof}
Let $f,g\in \text{PW}_{\Omega}$. It is clear that if $f=\pm g$, then $i)$ holds. Conversely, assume that $i)$ is true. Setting 
\begin{equation*}
 \phi=\frac{(-1)^\ell(2\pi i)^\ell}{c}\psi   
\end{equation*}
and following the same argument as in the proof of Theorem~\ref{thm:unquenessresultreal}, we can establish that $a_k^{-2\ell}|\mathcal{W}_{\phi}f(\frac{m}{4\Omega},a_k)|^2$ and $a_k^{-2\ell}|\mathcal{W}_{\phi}g(\frac{m}{4\Omega},a_k)|^2$ converge to $|f^{(\ell)}(\frac{m}{4\Omega})|^2$ and $|g^{(\ell)}(\frac{m}{4\Omega})|^2$, respectively, for every $m\in\mathbb{Z}$ as $k\to\infty$. Then, since $i)$ holds, we have that $$\left|f^{(\ell)}\left(\frac{m}{4\Omega}\right)\right|^2=\left|g^{(\ell)}\left(\frac{m}{4\Omega}\right)\right|^2, \quad \mbox{for all } m \in \mathbb{Z}.$$ Furthermore, by Proposition~\ref{prop:squareinpw}, we know that $|f^{(\ell)}|^2$ and $|g^{(\ell)}|^2$ belong to ${\rm PW}^1_{2\Omega}\subseteq {\rm PW}_{2\Omega}$. Thus, by the WSK sampling theorem it follows that
$$\left|f^{(\ell)}(x)\right|^2=\left|g^{(\ell)}(x)\right|^2, \quad \mbox{for all } x \in \R,$$
and consequently $|f^{(\ell)}(x)|=|g^{(\ell)}(x)|$  for all $x\in\R$. Finally, as in the proof of Theorem~\ref{thm:unquenessresultreal}, we can conclude that $f=\pm g$.
\end{proof}

\begin{oss}
More generally, we could replace the sampling set $(m/4\Omega)_{m\in\mathbb{Z}}$ in Theorem~\ref{thm:unquenessresultrealsampled} with any other sampling sequence for ${\rm PW}_{2\Omega}$. For instance, we could choose the sequence $(bm)_{m\in\mathbb{Z}}$ with $0<b\leq 1/4\Omega$. We refer to \cite{bruna2001sampling,seip2004interpolation} for a characterization of sampling sequences in terms of density properties.
\end{oss}

To make Theorem~\ref{thm:unquenessresultrealsampled} more palpable, we want to give a concrete example of a sampling set for which Theorem \ref{thm:unquenessresultrealsampled} implies uniqueness. To this end, we start recalling the definition of a wavelet system. Let $\psi\in L^1(\mathbb{R})$, and let $a>1$, $b>0$ be fixed. The set
\begin{equation}\label{eq:waveletsystem}
\mathcal{W}(\psi, a, b)=\{D_{a^n}T_{bm} \psi \}_{m,n\in\mathbb{Z}}=\{a^{-n}\psi(a^{-n}x-bm)\}_{m,n\in\mathbb{Z}}
\end{equation}
is called a \emph{wavelet system} with generator $\psi$ and parameters $a,b$. We observe that 
\[
\mathcal{W}_\psi f(a^{-n}bm,a^{-n})=\langle f , D_{a^n}T_{bm} \psi \rangle,\qquad m\in\mathbb{Z}, n\in\mathbb{N}.
\]
A typical choice for the parameters is the \emph{dyadic wavelet system} $\mathcal{W}(\psi, 2, 1)$. We refer to \cite[Chapter 12]{heil2011basis} and \cite[Chapter 3]{daub92} as classical references on wavelet systems and frames. Our next result follows from Theorem~\ref{thm:unquenessresultrealsampled} and shows that real-valued bandlimited signals can be uniquely recovered up to global sign from the absolute values of the wavelet coefficients with the wavelet system \eqref{eq:waveletsystem} for every choice of the parameters $a>1$ and $b>0$.
\begin{teo}[Cf.~Theorem \ref{thm:mainthmII}]
\label{thm:unquenessresultrealsampled2}
    Let $a > 1$, $b > 0$ and let $\psi \in L^1(\R)$ be a wavelet such that 
    \begin{equation*}
   \lim_{\xi\to0}\xi^{-\ell}\widehat{\psi}(\xi)=c\in\mathbb{C}\setminus\{0\}, 
\end{equation*}
for some $\ell\in\mathbb{N}$. Then, any real-valued bandlimited function $f \in L^2(\R)$ is uniquely determined  up to global sign by the phaseless wavelet coefficients
\[
\{\lvert \langle f , D_{a^n}T_{bm} \psi \rangle \rvert: m\in\mathbb{Z}, n\in\mathbb{N}\}.
\]
\end{teo}

\begin{proof}
    Let $f,g$ be real-valued bandlimited functions. Then, there exists $\Omega>0$ such that $f,g\in {\rm PW}_{\Omega}$. We suppose that 
    \[
    |\mathcal{W}_\psi f(a^{-n}bm,a^{-n})|^2=|\mathcal{W}_\psi g(a^{-n}bm,a^{-n})|^2,\qquad n\in\N,\ m\in\mathbb{Z},
    \]
    which correspond exactly to the squared magnitudes of the wavelet coefficients with wavelet system $\mathcal{W}(\psi, a, b)$. By Lemma~\ref{lem:wavelettransformbandlimited} together with Proposition~\ref{prop:squareinpw}, we know that $|\mathcal{W}_\psi f(\cdot,a^{-n})|^2\in {\rm PW}_{2\Omega}^1$, for every $n\in\mathbb{N}$. Furthermore, for every choice of the parameters $a > 1$ and $b > 0$, there exists $N\in\mathbb{N}$ such that for all $n > N$ it holds that $a^{-n}b<1/4\Omega$. Therefore, the WSK sampling theorem implies that 
    \[
    |\mathcal{W}_\psi f(b,a^{-n})|^2=|\mathcal{W}_\psi g(b,a^{-n})|^2,\qquad n > N,\ b\in\mathbb{R}. 
    \]
We can finally apply Theorem~\ref{thm:unquenessresultrealsampled} to conclude that $f$ coincides with $g$ up to global sign.
\end{proof}

\begin{oss}
    In view of Theorem~\ref{thm:unquenessresultrealsampled}, we observe that we could restrict the set of magnitude measurements to arbitrary fine scales, i.e. $a^{-n}$ with $n\geq N$ for every fixed $N\in\mathbb{N}$. 
\end{oss}

\section{Examples}\label{sec:examples}

\subsection{The Morlet wavelet}

%\begin{figure}
%    \centering
%    \begin{tikzpicture}
%        \begin{axis}[ymin = -0.5, height=10cm, axis lines = middle]
%            \addplot[domain=-2:2, samples=200, thick] {exp(-pi*x^2)*(cos(deg(pi*x)) - sin(deg(pi*x)))};
%            \addlegendentry{$f(x)$}
%            \addplot[domain=-2:2, samples=200, dashed, thick] {exp(-pi*x^2)*(cos(deg(pi*x)) + sin(deg(pi*x)))};
%            \addlegendentry{$g(x)$}
%        \end{axis}
%    \end{tikzpicture}
%    \caption{}
%    \label{fig:fandg}
%\end{figure}

Let $\xi_0\in\R\setminus\{0\}$. The Morlet wavelet, also known as the Gabor wavelet, is at the origin of the development of wavelet analysis. 
It was introduced by Grossmann and Morlet in \cite{grossmorlet84}. 
It is defined on the frequency side by the function
\[
\widehat{\psi}(\xi)=\pi^{-\frac{1}{4}}[e^{-(\xi-\xi_0)^2/2}-e^{-\xi^2/2}e^{-\xi_0^2/2}],\quad \xi\in\R.
\]
Its Fourier transform is a shifted Gaussian adjusted with a corrective term in order to have $\widehat{\psi}(0)=0$. The Morlet wavelet
\[
\psi(x)=\pi^{-\frac{1}{4}}[e^{i\xi_0x}-e^{-\xi_0^2/2}]e^{-x^2/2},\qquad x\in\R,
\]
is complex-valued but widely used for applications that involve only real-valued signals. By a direct computation, the Fourier transform $\widehat{\psi}(\xi)$ goes to zero as $\xi\to0$ with infinitesimal order 1. Indeed, using a Taylor expansion yields
\begin{align*}
\lim_{\xi\to0}\frac{\widehat{\psi}(\xi)}{\xi}&=\lim_{\xi\to0}\frac{\pi^{-\frac{1}{4}}[e^{-(\xi-\xi_0)^2/2}-e^{-\xi^2/2}e^{-\xi_0^2/2}]}{\xi}\\
&=\lim_{\xi\to0}\frac{\pi^{-\frac{1}{4}}[e^{-\xi_0^2/2}+\xi_0e^{-\xi_0^2/2}\xi-e^{-\xi_0^2/2}+e^{-\xi_0^2/2}\xi^2/2+o(\xi)]}{\xi}\\
&=\lim_{\xi\to0}\frac{\pi^{-\frac{1}{4}}[\xi_0e^{-\xi_0^2/2}\xi+o(\xi)]}{\xi}=\pi^{-\frac{1}{4}}\xi_0e^{-\xi_0^2/2}.
\end{align*}
Thus, $\psi$ satisfies the hypothesis of Theorem~\ref{thm:unquenessresultreal} with $\ell=1$. Therefore, all real-valued $f\in\text{PW}_{\Omega}$ can be recovered up to global sign from the
measurements $|\mathcal{W}_{\psi}f(b,a)|$, for $b\in\R$, $a\in\R_+$. Furthermore, by 
Theorem~\ref{thm:unquenessresultrealsampled2} we know that it is enough to know the magnitude of the wavelet transform $\mathcal{W}_\psi f$ for 
the samples $\{(a^{-n}bm,a^{-n}):m\in\mathbb{Z},\, n\in\N\}$.

\begin{figure}
    \centering
        \begin{tikzpicture}[scale=0.7]
        \begin{axis}[height=9cm, axis lines = middle]
            \addplot[domain=-3:3, samples=200, thick] {pi^(-1/4)*(cos(deg(5*x))-exp(-25/2))*exp(-x^2/2)};
            \addlegendentry{Re($\psi$)}
            \addplot[domain=-3:3, samples=200, dashed, thick] {pi^(-1/4)*sin(deg(5*x))*exp(-x^2/2)};
            \addlegendentry{Im($\psi$)}
        \end{axis}
    \end{tikzpicture}
     \qquad
    \begin{tikzpicture}[scale=0.7]
        \begin{axis}[ymin = -0.5, height=9cm, axis lines = middle]
            \addplot[domain=-3:10, samples=200, thick] {pi^(-1/4)*(exp(-(x-5)^2/2)-exp(-x^2/2)*exp(-25/2))};
            \addlegendentry{$\widehat{\psi}$}
        \end{axis}
    \end{tikzpicture}
  %  \caption{The Fourier transform of the Morlet wavelet for $\xi_0=5$. Observe that it is\\ not identically zero on the negative frequencies but it is numerically small.}
 %   \label{fig:fandg1}
    \caption{The Morlet wavelet $\psi$ for $\xi_0=5$ in time and frequency representation. Observe that the Fourier transform of $\psi$ is not identically zero on the negative frequencies but it is numerically small.}
    \label{fig:fandg1}
\end{figure}
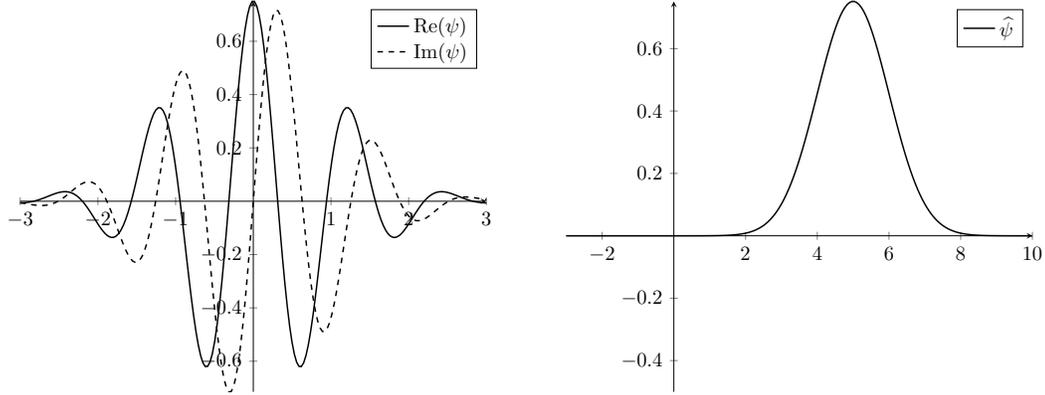

\subsection{The linear-chirp wavelet}

Another example of a complex-valued wavelet that satisfies our hypothesis is the linear-chirp wavelet, also called the chirplet. The idea to use chirps as wavelets was introduced in \cite{manhaykin91,manhaykin95}. We also refer to \cite{hol95} for a concise presentation. Let $\xi_0,\, \beta\in\R$. The chirplet is defined by windowing a linear chirp with a Gaussian: 
\[
\psi(x)=e^{i(\xi_0+\beta x/2)x}e^{-x^2/2}+\eta(x).
\]
Again, the corrective term $\eta$ is added in order to have zero mean.
Its Fourier transform is given by
\[
\widehat{\psi}(\xi)=\sqrt{\frac{2\pi}{1-i\beta}}e^{-(\xi-\xi_0)^2/2(1-i\beta)}+\widehat{\eta}(\xi).
\]
For instance, we may set 
\[
\widehat{\eta}(\xi)=-\sqrt{\frac{2\pi}{1-i\beta}} e^{-\xi_0^2/2(1-i\beta)}e^{-\xi^2/2}
\]
and using a Taylor expansion, we obtain
\begin{align*}
\lim_{\xi\to0}\frac{\widehat{\psi}(\xi)}{\xi}&=\lim_{\xi\to0}\sqrt{\frac{2\pi}{1-i\beta}}\frac{e^{-(\xi-\xi_0)^2/2(1-i\beta)}-e^{-\xi_0^2/2(1-i\beta)}e^{-\xi^2/2}}{\xi}\\
%&=\lim_{\xi\to0}\sqrt{\frac{2\pi}{1-i\beta}}\frac{e^{-\xi_0^2/2(1-i\beta)}\%xi_0\xi/(1-i\beta)+e^{-\xi_0^2/2(1-i\beta)}\xi^2/2+o(\xi)}{\xi}\\
&=\lim_{\xi\to0}\sqrt{\frac{2\pi}{1-i\beta}}\frac{e^{-\xi_0^2/2(1-i\beta)}\xi_0\xi/(1-i\beta)+o(\xi)}{\xi}
=\frac{\sqrt{2\pi}}{(1-i\beta)^{3/2}}e^{-\xi_0^2/2(1-i\beta)}\xi_0.
\end{align*}
 This shows that $\psi$ satisfies the hypothesis of Theorems~\ref{thm:unquenessresultreal}, \ref{thm:unquenessresultrealsampled} and \ref{thm:unquenessresultrealsampled2} with $\ell=1$.
 
 \begin{figure}
    \centering
    \begin{tikzpicture}[scale=0.7]
        \begin{axis}[height=9cm, axis lines = middle]
            \addplot[domain=-3:3, samples=200, thick] {cos(deg(5*x+x^2/2))*exp(-x^2/2)-pi^(1/2)*2^(1/4)*exp(-x^2/2)*exp(-25/4)*(cos(deg(pi/8))*cos(deg(-25/4))-sin(deg(pi/8))*sin(deg(-25/4)))};
            \addlegendentry{Re($\psi$)}
            \addplot[domain=-3:3, samples=200, dashed, thick] {sin(deg(5*x+x^2/2))*exp(-x^2/2)-pi^(1/2)*2^(1/4)*exp(-x^2/2)*exp(-25/4)*(cos(deg(pi/8))*sin(deg(-25/4))+sin(deg(pi/8))*cos(deg(-25/4)))};
            \addlegendentry{Im($\psi$)}
        \end{axis}
    \end{tikzpicture}
  %  \caption{The linear-chirp wavelet for $\beta=1$ and $\xi_0=5$ in time representation.}
 %   \label{fig:fandg3}
\qquad
    \begin{tikzpicture}[scale=0.7]
        \begin{axis}[ymin = -1, height=9cm, axis lines = middle]
            \addplot[domain=-3:10, samples=200, thick] {pi^(1/2)*2^(1/4)*exp(-(x-5)^2/4)*(cos(deg(pi/8))*cos(deg(-(x-5)^2/4))-sin(deg(pi/8))*sin(deg(-(x-5)^2/4)))-pi^(1/2)*2^(1/4)*exp(-x^2/2)*exp(-25/4)*(cos(deg(pi/8))*cos(deg(-25/4))-sin(deg(pi/8))*sin(deg(-25/4)))};
            \addlegendentry{Re($\widehat{\psi}$)}
            \addplot[domain=-3:10, samples=200, dashed, thick] {pi^(1/2)*2^(1/4)*exp(-(x-5)^2/4)*(cos(deg(pi/8))*sin(deg(-(x-5)^2/4))+sin(deg(pi/8))*cos(deg(-(x-5)^2/4)))-pi^(1/2)*2^(1/4)*exp(-x^2/2)*exp(-25/4)*(cos(deg(pi/8))*sin(deg(-25/4))+sin(deg(pi/8))*cos(deg(-25/4)))};
            \addlegendentry{Im($\widehat{\psi}$)}
        \end{axis}
    \end{tikzpicture}
    \caption{The linear-chirp wavelet for $\beta=1$ and $\xi_0=5$ in time and frequency representation.}
    \label{fig:fandg3}
\end{figure}
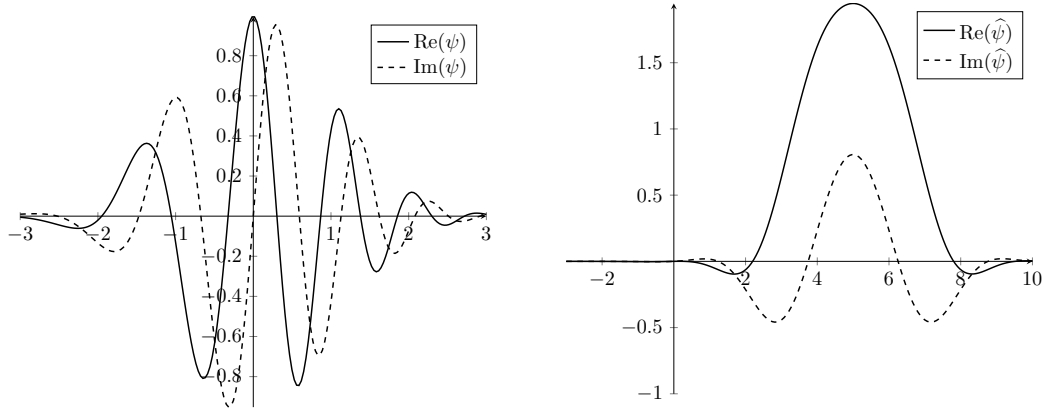

\section{Sampling Cauchy wavelet transform magnitudes}
\label{sec:cwt}

% \important{\textbf{Matthias:} I have completely rewritten this section. I did not color all of it in blue because I think it is more readable that way.

% I will rewrite certain parts of this section once more after the introduction has been rewritten. In this way, there are not too many redundant statements.

% We will need to decide what lingo we use in this section. See the blue remark in the middle of it.}

\subsection{Introduction}
Our main uniqueness result for phase retrieval from wavelet magnitude samples (Theorem
\ref{thm:unquenessresultrealsampled}) is not applicable to so-called progressive
wavelets which are wavelets that only have positive frequencies. 

This observation is not surprising in light of the fact that real-valued signals $f$ are not uniquely determined (up to global phase) by wavelet transform magnitude measurements $\lvert \cW_\psi f \rvert$ for progressive wavelets $\psi$ (see Remark \ref{rem:rimas_counterexample} in Section \ref{sec:main_results}). It does, however, raise the following question: 
\begin{itemize}
	\item[(Q)] Is there a class of signals which can be recovered (up to global phase) from wavelet transform magnitude measurements with progressive mother wavelets?
\end{itemize}
In general, this question is hard to answer. An elegant partial answer is however given in \cite{mallat2015phase}.

The authors of \cite{mallat2015phase} consider the so-called \emph{Cauchy wavelet} given by
\begin{equation}\label{eq:cauchywavelet}
    \widehat \psi(\xi) = \rho(\xi) \xi^p \mathrm{e}^{-\xi} \boldsymbol{1}_{\xi > 0}, \qquad \xi \in \R,
\end{equation}
where $p > 0$ and $\rho \in L^\infty(\R)$ is such that $\rho(a\xi) = \rho(\xi)$, for
a.e.~$\xi \in \R$, and $\rho(\xi) \neq 0$, for all $\xi \in \R$. Using tools from the theory of entire functions, they show that the class of analytic signals may be recovered uniquely (up to global phase) from the magnitude of the Cauchy wavelet transform. Analytic signals are functions $f \in L^2(\R)$ which have no negative frequencies. To be precise, they show the following theorem.
\begin{figure}
    \centering
    \begin{tikzpicture}[scale=0.7]
        \begin{axis}[ymin = -0.1, height=9cm, axis lines = middle]
            \addplot[domain=0:10, samples=200, thick] {x^2*exp(-x)};
            \addlegendentry{$\widehat{\psi}$}
        \end{axis}
    \end{tikzpicture}
    \caption{The Cauchy wavelet $\widehat \psi(\xi) =\xi^2 \mathrm{e}^{-\xi} \boldsymbol{1}_{\xi > 0}$.}
    \label{fig:fandg5}
\end{figure}
% \mcwadd{\textbf{Matthias:} A linguistic question: Now that we say progressive wavelets rather than analytic wavelets: Should we also say progressive signals rather than analytic signals? Or should we talk about the analytic part of signals?} \raadd{Rima: In my opinion it is fine to talk about progressive wavelets but analytic signals. I would only suggest to change the use of "analytic part" to "analytic representation" as I believe it to be more standard.}

\begin{teo}[Corollary 2.2 in \cite{mallat2015phase}, p.~1259]
    \label{thm:mallatsemidiscretesampling}
    Let $a > 1$ and let $\psi$ be the Cauchy wavelet defined as in equation
    \eqref{eq:cauchywavelet}. Let, moreover, $f,g \in L^2(\R)$ be such that for some $j,k \in
    \mathbb{Z}$, with $j \neq k$,
    \[
        \left\lvert \cW_\psi f\left(\cdot,a^j\right) \right\rvert
            = \left\lvert \cW_\psi g\left(\cdot,a^j\right) \right\rvert
        \qquad \mbox{and} \qquad
        \left\lvert \cW_\psi f\left(\cdot,a^k\right) \right\rvert
            = \left\lvert \cW_\psi g\left(\cdot,a^k\right) \right\rvert.
    \]
    We denote by $f_+$ and $g_+$ the analytic representations of $f$ and $g$ which are defined through 
    the equations 
    \[
        \widehat{f_+}(\xi) = 2 \widehat f (\xi) \boldsymbol{1}_{\xi > 0} \qquad \mbox{and} \qquad
        \widehat{g_+}(\xi) = 2 \widehat g (\xi) \boldsymbol{1}_{\xi > 0},
    \]
    for $\xi \in \R$. Then, there exists an $\alpha \in \R$ such that 
    \[
        f_+ = \mathrm{e}^{\mathrm{i} \alpha} g_+.
    \]
\end{teo}

Note that the above is more than a simple uniqueness theorem for phase retrieval from Cauchy wavelet transform magnitude measurements of analytic signals. It is, in fact, a uniqueness result for the \emph{semi-discrete} wavelet frame. Even more, the Cauchy wavelet magnitudes are assumed to agree on two scales only. It does therefore stand to reason that further restricting the signal class to analytic bandlimited signals should allow us to come up with a full sampling result for the Cauchy wavelet transform.

In the following, we will assume that the function $\rho \in L^\infty(\R)$ used in the definition of the Cauchy wavelet $\psi$ is such that $\psi \in L^1(\R)$. This is a natural assumption as mother wavelets are usually assumed to be in $L^1(\R)$. Moreover, there is a wide variety of $\rho \in L^\infty(\R)$ which satisfy this assumption (see Remark \ref{rem:rima_cwl1}). We want to stress, however, that this assumption is not necessary for our arguments to work and is made purely to simplify the mathematical exposition. Indeed, by the definition of the Cauchy wavelet \eqref{eq:cauchywavelet}, one can see immediately that $\psi \in L^2 \cap L^\infty(\R)$. Therefore, we may replace our subsequent use of the WSK sampling theorem by the use of classical sampling theory in the Bernstein space $B_{2\Omega}$ to obtain a sampling result for more general $\rho \in L^\infty(\R)$ at a slightly finer sampling density in frequency. 

\begin{oss}
\label{rem:rima_cwl1}
    One can show that if $\rho \in L^\infty(\R)$ is continuous and satisfies 
    \begin{equation*}
        \lvert \rho'(\xi) \rvert \lesssim \mathrm{e}^{\xi/2} \qquad \mbox{and}
        \qquad \lvert \rho''(\xi) \rvert \lesssim \mathrm{e}^{\xi/2},
    \end{equation*}
    for all $\xi \in \R$, then the Cauchy wavelet $\psi$ defined by \eqref{eq:cauchywavelet} is in $L^1(\R)$. In particular, if $\rho$ is a constant function, then $\psi \in L^1(\mathbb{R})$.
\end{oss}

\subsection{The sampling result for analytic signals}

We remind the reader of two pertinent results stated earlier in this manuscript: First, the wavelet transform $\cW_\psi f(\cdot,a)$ of a bandlimited signal $f$ is bandlimited itself, for $a \in \R_+$ (see Lemma \ref{lem:wavelettransformbandlimited}). Secondly, bandlimitedness carries over from any function to its squared absolute value. These two insights combined yield the following corollary. 

\begin{cor}
    \label{cor:objectisinpaleywienerspace}
    Let $\Omega > 0$. If $f \in \mathrm{PW}_\Omega$ and $\psi \in L^1(\R)$, then
    $\lvert \cW_\psi f(\cdot,a) \rvert^2 \in \mathrm{PW}_{2\Omega}^1 \subset \mathrm{PW}_{2\Omega}$, for all $a \in \R_+$.
\end{cor}

What remains is to combine Theorem \ref{thm:mallatsemidiscretesampling} with the classical WSK sampling theorem (Theorem \ref{thm:wsk}). Thereby, we obtain the following sampling result for the recovery of analytic signals.

\begin{teo}
    \label{thm:cwtmaintheorem}
    Let $\Omega > 0$, $a > 1$ and let $\psi\in L^1(\mathbb{R})$ be as in equation
    \eqref{eq:cauchywavelet} with $\rho \in L^\infty(\R)$. Then, the following are equivalent for $f,g \in \mathrm{PW}_\Omega$:
    \begin{enumerate}
        \item[i)] For all $k \in \mathbb{Z}$,
        \[
            \left\lvert \cW_\psi f\left(\frac{k}{4\Omega},1\right) \right\rvert
                = \left\lvert \cW_\psi g\left(\frac{k}{4\Omega},1\right) \right\rvert
            \mbox{ and }
            \left\lvert \cW_\psi f\left(\frac{k}{4\Omega},a\right) \right\rvert
                = \left\lvert \cW_\psi g\left(\frac{k}{4\Omega},a\right) \right\rvert.
        \]
        \item[ii)] $f_+ = \mathrm{e}^{\mathrm{i} \alpha} g_+$, for some $\alpha \in \R$.
    \end{enumerate}
\end{teo}

\begin{proof}
    It is obvious that $ii)$ implies $i)$. Now, suppose that $i)$ holds. By assumption, we have that $\psi \in L^1(\R)$. Therefore, we may apply Corollary \ref{cor:objectisinpaleywienerspace} to see that $\lvert \cW_\psi f(\cdot,a^j) \rvert^2$ as well as $\lvert \cW_\psi g(\cdot,a^j) \rvert^2$ are in $\mathrm{PW}_{2\Omega}$, for $j \in \{0,1\}$. Hence, it follows from $i)$ along with the WSK sampling theorem that
    \[
        \left\lvert \cW_\psi f\left(\cdot,1\right) \right\rvert
            = \left\lvert \cW_\psi g\left(\cdot,1\right) \right\rvert
        \qquad \mbox{and} \qquad
        \left\lvert \cW_\psi f\left(\cdot,a\right) \right\rvert
            = \left\lvert \cW_\psi g\left(\cdot,a\right) \right\rvert.
    \]
    Finally, Theorem \ref{thm:mallatsemidiscretesampling} implies that the analytic representations $f_+$ and $g_+$ of $f$ and $g$ satisfy $f_+ = \mathrm{e}^{\mathrm{i} \alpha} g_+$, for some $\alpha \in \R$.
\end{proof}

\begin{oss}
    Note that the sampling set $\{ (k/4\Omega,a^j) \,\mid\, j=0,1,~k \in \mathbb{Z} \}$ in Theorem \ref{thm:cwtmaintheorem} can be replaced by a multitude of different sampling sets:
    
    In scale, we might sample at any two elements of $a^\mathbb{Z}$ as is evident from Theorem \ref{thm:mallatsemidiscretesampling}. In addition, one might show that Theorem
    \ref{thm:mallatsemidiscretesampling} continues to hold for $a^j$ replaced by $a_0$ and $a^k$ replaced by $a_1$, for all $0 < a_0 < a_1$, provided that $\rho(\xi) = \rho(a_0 \xi) = \rho(a_1 \xi)$, for a.e.~$\xi \in \R$.

    In time, we can replace the uniform sampling by any sampling sequence for $\mathrm{PW}_{2\Omega}$, see \cite{bruna2001sampling,seip2004interpolation}.
\end{oss}

% \paragraph{Open problems} We would like to conclude our paper mentioning some questions that are left open. It does not seem possible to generalise the Theorems \ref{thm:mainthmI} and \ref{thm:mainthmII} to \emph{complex-valued} bandlimited functions (or even to all of $L^2(\R)$) by using the techniques we employ in our proofs. We are of the opinion that such a generalisation would be highly interesting from the point of view of pure mathematics (while potentially being less riveting for application purposes). 

% It is of course also interesting to study whether one can reduce the amount of measurements needed to recover real-valued bandlimited signals from sampled wavelet transform magnitudes with respect to Theorem \ref{thm:mainthmII}. In particular, it seems conceivable that magnitude information on a finite number of scales (potentially as few as two scales) is sufficient for recovery up to global sign. The result on the Cauchy wavelet transform provides some evidence in this direction: Indeed, one can prove a uniqueness result for the analytic part of bandlimited signals by only making use of two scales (cf.~Theorem \ref{thm:cwtmaintheorem}).
% \vspace{0.5cm}

\paragraph{Acknowledgements}\quad The authors would like to acknowledge funding through SNF Grant 200021 184698. 
%F. Bartolucci is part of the Machine Learning Genoa Center (MaLGa).

\bibliographystyle{abbrv}
\bibliography{bibliowavelet}

\end{document}